\documentclass[a4paper,12pt]{article}

\usepackage{amsmath,amsthm,amssymb}
\usepackage{txfonts}
\usepackage[top=30truemm,bottom=30truemm,left=25truemm,right=25truemm]{geometry}

\theoremstyle{plain}
\newtheorem{definition}{Definition}
\newtheorem{thm}[definition]{Theorem}

\newtheorem{lem}[definition]{Lemma}

\def\F#1#2#3#4{F\biggl(\genfrac..{0pt}{}{{#1},\,{#2}}{#3}\,;#4\biggr)}
\def\f#1#2#3#4{f\biggl(\genfrac..{0pt}{}{{#1},\,{#2}}{#3}\,;#4\biggr)}
\def\y#1#2#3#4#5{f_{#1}\biggl(\genfrac..{0pt}{}{{#2},\,{#3}}{#4}\,;#5\biggr)}
\def\c#1#2#3#4#5#6#7{\biggl(\genfrac..{0pt}{}{{#1},\,{#2}}{#3}\,;\genfrac..{0pt}{}{{#4},\,{#5}}{#6}\,;#7\biggr)}
\def\Q#1#2#3#4#5#6#7{Q\biggl(\genfrac..{0pt}{}{{#1},\,{#2}}{#3}\,;\genfrac..{0pt}{}{{#4},\,{#5}}{#6}\,;#7\biggr)}
\def\R#1#2#3#4#5#6#7{R\biggl(\genfrac..{0pt}{}{{#1},\,{#2}}{#3}\,;\genfrac..{0pt}{}{{#4},\,{#5}}{#6}\,;#7\biggr)}
\def\q#1#2#3#4#5#6#7{q\biggl(\genfrac..{0pt}{}{{#1},\,{#2}}{#3}\,;\genfrac..{0pt}{}{{#4},\,{#5}}{#6}\,;#7\biggr)}
\def\G{\Gamma}
\def\vec#1{\mbox{\boldmath $#1$}}

\DeclareMathOperator{\map}{Map}

\makeatletter
\@addtoreset{equation}{section}

\makeatother

\begin{document}
\title{Symmetries of Coefficients of Three-Term Relations for the Hypergeometric Series}
\author{Yuka Yamaguchi}
\date{\today}

\maketitle

\begin{abstract}
Any three hypergeometric series whose respective parameters, $a,\ b$ and $c$, 
differ by integers satisfy a linear relation 
with coefficients that are rational functions of $a,\ b,\ c$ and the variable $x$. 
These relations are called three-term relations. 
This paper shows that the coefficients of three-term relations have 
properties called symmetries, 
and gives explicit formulas describing the symmetries. 
\end{abstract}



\section{\bf{Introduction}}
The hypergeometric series is defined by 
\begin{align*}
F (a, b, c\,; x) = \F{a}{b}{c}{x} := \sum_{n = 0}^{\infty} \frac{(a)_{n} (b)_{n}}{(c)_{n} (1)_{n}} x^{n}. 
\end{align*}
Here, $(\alpha)_{n}$ denotes $\G(\alpha + n) / \G(\alpha)$, 
which equals $\alpha (\alpha + 1) \dotsm (\alpha + n - 1)$ for any positive integer $n$. 
It is assumed that $c$ is such that the denominator factor $(c)_{n}$ is never zero. 

As mentioned in \cite[Section~$2.5$, p.$94$]{AAR}, it is known that for any triples of integers $(k, l, m)$ and $(k', l', m')$, 
the three hypergeometric series 
\begin{align*}
\F{a + k}{b + l}{c + m}{x}, \quad \F{a + k'}{b + l'}{c + m'}{x}, \quad \F{a}{b}{c}{x}
\end{align*}
satisfy a linear relation with coefficients that are rational functions of $a,\ b,\ c$ and $x$. 
We call such a relation a three-term relation.  
Gauss obtained the three-term relations in the cases 
\begin{align*}
(k, l, m), (k', l', m') \in \left\{(1, 0, 0), (-1, 0, 0), (0, 1, 0), (0, -1, 0), (0, 0, 1), (0, 0, -1) \right\}, 
\end{align*}
where $(k, l, m) \neq (k', l', m')$; 
thus, there are $\binom{6}{2} = 15$ pairs of $(k, l, m)$ and $(k', l', m')$. 
See \cite[Chapter~$4$, p.$71$]{Rain} for the 15 three-term relations obtained by Gauss. 

We consider the three-term relations of the following form: 
\begin{align}\label{3kou}
\F{a + k}{b + l}{c + m}{x} 
= Q \cdot \F{a + 1}{b + 1}{c + 1}{x} + R \cdot \F{a}{b}{c}{x}. 
\end{align}
Note that the pair $(Q, R)$ of rational functions of $a,\ b,\ c$ and $x$ 
is uniquely determined by $(k, l, m)$ (cf. \cite[Chapter~$6$, Section~$23$]{Poole}). 
Ebisu \cite[Section~$2. 3$]{Eb2} noticed that the coefficient $Q$ in $(\ref{3kou})$ has 48 symmetries, and using these symmetries, he gave many special values of the hypergeometric series. 
Unfortunately, Ebisu presented only two explicit formulas \cite[(2.7), (2.11)]{Eb2} 
as examples of $Q$'s symmetries. 
On the other hand, Vid\=unas considered the three-term relations 
of the form
\begin{align*}
\F{a + k}{b + l}{c + m}{x} 
= \tilde{Q} \cdot \F{a + 1}{b}{c}{x} + \tilde{R} \cdot \F{a}{b}{c}{x}, 
\end{align*}
and gave an explicit formula describing $\tilde{Q}$'s symmetry \cite[p. 509, $(11)$]{Vidunas}. 
We will see that $Q$ also has the same symmetry. 

In this paper, combining the 48 symmetries noticed by Ebisu and 
another symmetry the counterpart of one obtained by Vid\=unas, 
we show that $Q$ has 96 symmetries. 
In addition, we give a relation between $Q$ and $R$, and using the relation, 
we derive 96 symmetries of $R$ from the 96 symmetries of $Q$. 

To avoid ambiguity, we first define the notion of a symmetry of $Q$ and $R$. 
For the parameters $a$, $b$, $c$ and the variable $x$, 
let $S_{a b c}$, $S_{x}$ and $S$ be the sets defined by 
\begin{align*}
S_{a b c} &:= \left\{n_{0} + n_{1} a + n_{2} b + n_{3} c \mid n_{i} \in \mathbb{Z} \right\}, \\
S_{x} &:= \left\{x, \frac{x}{x - 1}, 1 - x, \frac{x - 1}{x}, \frac{1}{x}, \frac{1}{1 - x} \right\}, \\
S &:= \left\{\left(\genfrac..{0pt}{}{k,\, l}{m}\,;\genfrac..{0pt}{}{\alpha_{1},\, \alpha_{2}}{\alpha_{3}}\,; \beta\right) 
\; \middle\vert \; 
k, l, m \in \mathbb{Z},\; \alpha_{i} \in S_{a b c},\; \beta \in S_{x} \right\}, 
\end{align*}
and let $T$ be the set of all rational functions of $a,\ b,\ c$ and $x$. 
Also, let $\map(S, T)$ denote the set of all functions $P : S \to T$. 
Then, $Q$ and $R$ can be regarded as elements of $\map(S, T)$; namely, 
\begin{align*}
Q = \Q{k}{l}{m}{a}{b}{c}{x}, \quad 
R = \R{k}{l}{m}{a}{b}{c}{x} \: \in \map(S, T). 
\end{align*}
We define the notion of a symmetry of elements of $\map(S, T)$ as follows: 
\begin{definition}
Let $G$ be a group that acts on $\map(S, T)$, and 
take any $\varphi \in G$ and $P \in \map(S, T)$. 
If for any $k, l, m \in \mathbb{Z}$, 
there exist $\alpha_{1}, \ldots, \alpha_{n} \in S_{a b c}$ and 
$i_{1}, \ldots, i_{n}, j_{1}, j_{2}, j_{3} \in \mathbb{Z}$ satisfying 
\begin{align*}
P \c{k}{l}{m}{a}{b}{c}{x} 
= (\alpha_{1})_{i_{1}} (\alpha_{2})_{i_{2}} \dotsm (\alpha_{n})_{i_{n}} (-1)^{j_{1}} x^{j_{2}} (1 - x)^{j_{3}} 
\left(\varphi P \right) \c{k}{l}{m}{a}{b}{c}{x}, 
\end{align*}
then we say that $P$ has a symmetry under $\varphi$. 
If $P$ has a symmetry under an arbitrary $\varphi \in G$, 
then we say that $P$ has symmetries under the action of $G$. 
\end{definition}

We introduce a transformation group $G$ and define an action of $G$ on $\map(S, T)$. 
Let $G$ be the group generated by the following four mappings 
so that $G$ acts on $S$: 
\begin{align*}
\sigma_{0} : \: &\c{k}{l}{m}{a}{b}{c}{x} \mapsto \c{-k}{-l}{-m}{a + k}{b + l}{c + m}{x}, \allowdisplaybreaks \\
\sigma_{1} : \: &\c{k}{l}{m}{a}{b}{c}{x} \mapsto \c{m - k}{l}{m}{c - a}{b}{c}{\frac{x}{x - 1}}, \allowdisplaybreaks \\
\sigma_{2} : \: &\c{k}{l}{m}{a}{b}{c}{x} \mapsto \c{k}{l}{k + l - m}{a}{b}{a + b + 1 - c}{1 - x}, \allowdisplaybreaks \\
\sigma_{3} : \: &\c{k}{l}{m}{a}{b}{c}{x} \mapsto \c{l}{k}{m}{b}{a}{c}{x}, 
\end{align*}
where group operation is defined as the composition of elements in $G$. 
We will see that $G$ is isomorphic to 
$\mathbb{Z} / 2 \mathbb{Z} \times \left(S_{3} \ltimes \left(\mathbb{Z} / 2 \mathbb{Z} \right)^{3} \right)$; 
thus, the order of $G$ equals 96 (see Lemma~$\ref{lem:G}$).  
We define an action of $G$ on $\map(S, T)$ by
$\left(\sigma P \right) (\vec{z}) := P \left(\sigma^{-1} \vec{z} \right)$, 
where $\sigma \in G$, $P \in \map(S, T)$ and $\vec{z} \in S$. 

The following theorem provides $Q$'s symmetries. 
\begin{thm}\label{Q-sym}
The coefficient $Q$ of $(\ref{3kou})$ has symmetries under the action of $G$; 
thus, $Q$ has 96 symmetries. 
In fact, $Q$ has the following symmetries: 
\begin{align}
\Q{k}{l}{m}{a}{b}{c}{x} 
&= \frac{(c + 1)_{m} (c)_{m} (-1)^{m - k - l - 1} x^{-m} (1 - x)^{m - k - l}}{(a + 1)_{k} (b + 1)_{l} (c - a)_{m - k} (c - b)_{m - l}} 
(\sigma_{0} Q) \c{k}{l}{m}{a}{b}{c}{x}, \label{(sig0)Q} \allowdisplaybreaks \\
\Q{k}{l}{m}{a}{b}{c}{x} 
&= - \frac{a}{c - a} (1 - x)^{2 - l} 
(\sigma_{1} Q) \c{k}{l}{m}{a}{b}{c}{x}, \label{(sig1)Q} \allowdisplaybreaks \\
\Q{k}{l}{m}{a}{b}{c}{x} 
&= \frac{(c + 1)_{m - 1} (c - a - b - 1)_{m + 1 - k - l}}{(c - a)_{m - k} (c - b)_{m - l}} 
(\sigma_{2} Q) \c{k}{l}{m}{a}{b}{c}{x}, \label{(sig2)Q} \allowdisplaybreaks \\
\Q{k}{l}{m}{a}{b}{c}{x} 
&= (\sigma_{3} Q) \c{k}{l}{m}{a}{b}{c}{x}. \label{(sig3)Q}
\end{align}
In addition, combining these formulas, 
we are able to obtain the other 92 explicit formulas describing $Q$'s symmetries. 
\end{thm}
The identity $(\ref{(sig0)Q})$ is counterpart of the symmetry of $\tilde{Q}$ given 
in \cite[p. 509, $(11)$]{Vidunas}. 
The explicit formulas in \cite[(2.7), (2.11)]{Eb2} describe 
the symmetries of $Q$ under $\sigma_{1} \sigma_{3} \sigma_{1} \sigma_{3}$ and 
$\sigma_{3} \sigma_{1} \sigma_{3}$, respectively. 

The following lemma is used to derive $R$'s symmetries from the $Q$'s symmetries. 
\begin{lem}\label{R=Q'}
The coefficients of $(\ref{3kou})$ satisfy the following relation: 
\begin{align*}
\R{k}{l}{m}{a}{b}{c}{x} 
= \frac{c (c + 1)}{(a + 1) (b + 1) x (1 - x)} \Q{k - 1}{l - 1}{m - 1}{a + 1}{b + 1}{c + 1}{x}. 
\end{align*}
\end{lem}

We introduce a transformation group $\tilde{G}$ and define an action of $\tilde{G}$ on $\map(S, T)$. 
Let $\tilde{G}$ be the group generated by the following four mappings so that $\tilde{G}$ acts on $S$: 
\begin{align*}
\tilde{\sigma}_{0} := \tau \sigma_{0} \tau^{-1} : 
\: &\c{k}{l}{m}{a}{b}{c}{x} \mapsto \c{2 - k}{2 - l}{2 - m}{a + k - 1}{b + l - 1}{c + m - 1}{x}, \allowdisplaybreaks \\
\tilde{\sigma}_{1} := \tau \sigma_{1} \tau^{-1} : 
\: &\c{k}{l}{m}{a}{b}{c}{x} \mapsto \c{m + 1 - k}{l}{m}{c - a - 1}{b}{c}{\frac{x}{x - 1}}, \allowdisplaybreaks \\
\tilde{\sigma}_{2} := \tau \sigma_{2} \tau^{-1} : 
\: &\c{k}{l}{m}{a}{b}{c}{x} \mapsto \c{k}{l}{k + l - m}{a}{b}{a + b + 1 - c}{1 - x}, \allowdisplaybreaks \\
\tilde{\sigma}_{3} := \tau \sigma_{3} \tau^{-1} : 
\: &\c{k}{l}{m}{a}{b}{c}{x} \mapsto \c{l}{k}{m}{b}{a}{c}{x}, 
\end{align*}
where $\sigma_{i}$ ($i = 0, 1, 2, 3$) are the mappings defined in the above, $\tau$ is the mapping defined by 
\begin{align*}
\tau : \c{k}{l}{m}{a}{b}{c}{x} \mapsto \c{k + 1}{l + 1}{m + 1}{a - 1}{b - 1}{c - 1}{x}, 
\end{align*}
and group operation is defined as the composition of elements in $\tilde{G}$. 
It immediately follows from this definition that $\tilde{G}$ is isomorphic to $G$; 
thus, the order of $\tilde{G}$ also equals 96. 
We define an action of $\tilde{G}$ on $\map(S, T)$ 
in the same way as the action of $G$ on $\map(S, T)$. 

The following theorem provides $R$'s symmetries. 
\begin{thm}\label{R-sym}
The coefficient $R$ of $(\ref{3kou})$ has symmetries under the action of $\tilde{G}$; 
thus, $R$ has 96 symmetries. 
In fact, $R$ has the following symmetries: 
\begin{align}
\R{k}{l}{m}{a}{b}{c}{x} 
&= \frac{(c + 1)_{m - 1} (c)_{m - 1}}{(a + 1)_{k - 1} (b + 1)_{l - 1} (c - a)_{m - k} (c - b)_{m - l}} \nonumber \\
&\quad \times (-1)^{m - k - l} x^{1 - m} (1 - x)^{m + 1 - k - l} 
(\tilde{\sigma}_{0} R) \c{k}{l}{m}{a}{b}{c}{x}, \label{(tsig0)R} \allowdisplaybreaks \\
\R{k}{l}{m}{a}{b}{c}{x} 
&= (1 - x)^{- l} (\tilde{\sigma}_{1} R) \c{k}{l}{m}{a}{b}{c}{x}, \label{(tsig1)R} \allowdisplaybreaks \\
\R{k}{l}{m}{a}{b}{c}{x} 
&= \frac{(c)_{m} (c - a - b)_{m - k - l}}{(c - a)_{m - k} (c - b)_{m - l}} 
(\tilde{\sigma}_{2} R) \c{k}{l}{m}{a}{b}{c}{x}, \label{(tsig2)R} \allowdisplaybreaks \\
\R{k}{l}{m}{a}{b}{c}{x} 
&= (\tilde{\sigma}_{3} R) \c{k}{l}{m}{a}{b}{c}{x}. \label{(tsig3)R}
\end{align}
In addition, combining these formulas, 
we are able to obtain the other 92 explicit formulas describing $R$'s symmetries. 
\end{thm}

\section{\bf{Proof of Theorem~$\ref{Q-sym}$}}
After characterizing structure of $G$, we prove Theorem~$\ref{Q-sym}$. 

Let $\sigma_{4} := \sigma_{1} \sigma_{3} \sigma_{1} \sigma_{3}$ and $\sigma_{5} := \sigma_{2} \sigma_{4} \sigma_{2} \sigma_{4} \sigma_{3}$ to make them become 
\begin{align*}
\sigma_{4} : \: &\c{k}{l}{m}{a}{b}{c}{x} \mapsto \c{m - k}{m - l}{m}{c - a}{c - b}{c}{x}, \\
\sigma_{5} : \: &\c{k}{l}{m}{a}{b}{c}{x} \mapsto \c{-k}{-l}{-m}{1 - a}{1 - b}{2 - c}{x}. 
\end{align*}
Then, we obtain the following lemma.  
\begin{lem}\label{lem:G}
The structure of $G$ is identified as 
\begin{align*}
G = \langle \sigma_{0} \rangle \times 
\left(\langle \sigma_{1}, \sigma_{2} \rangle \ltimes \left(\langle \sigma_{3} \rangle \times \langle \sigma_{4} \rangle \times \langle \sigma_{5} \rangle \right) \right) 
\cong \mathbb{Z} / 2 \mathbb{Z} \times \left(S_{3} \ltimes \left(\mathbb{Z} / 2 \mathbb{Z} \right)^{3} \right), 
\end{align*}
where $S_{3}$ is the symmetric group of degree $3$; 
thus, the order of $G$ equals $2 \cdot 3! \cdot 2^{3} = 96$. 
\end{lem}
\begin{proof}
First, $G = \langle \sigma_{0} \rangle \times \langle \sigma_{1}, \sigma_{2}, \sigma_{3} \rangle$ holds. 
Next, 
since $\langle \sigma_{3}, \sigma_{4}, \sigma_{5} \rangle$ is normal 
in $\langle \sigma_{1}, \sigma_{2}, \sigma_{3} \rangle$ and satisfies 
$\langle \sigma_{1}, \sigma_{2} \rangle \cap \langle \sigma_{3}, \sigma_{4}, \sigma_{5} \rangle = \left\{{\rm Id}_{G} \right\}$, where ${\rm Id}_{G}$ denotes the identity element of $G$, 
it holds that $\langle \sigma_{1}, \sigma_{2}, \sigma_{3} \rangle \cong \langle \sigma_{1}, \sigma_{2} \rangle \ltimes \langle \sigma_{3}, \sigma_{4}, \sigma_{5} \rangle$. 
Finally, from 
$\sigma_{i}^{2} = {\rm Id}_{G}$ ($0 \leq i \leq 5$), 
$\sigma_{1} \sigma_{2} = \sigma_{2} \sigma_{1} \sigma_{2} \sigma_{1}$ and 
$\sigma_{i} \sigma_{j} = \sigma_{j} \sigma_{i}$ ($3 \leq i, j \leq 5$), 
the proof of the lemma is complete. 
\end{proof}

We prove Theorem~$\ref{Q-sym}$. 
From the uniqueness of analytic continuation, 
it is sufficient to show that $(\ref{(sig0)Q})\mbox{--}(\ref{(sig3)Q})$ hold for $\lvert x \rvert < 1/2$; thus, below, we assume $\lvert x \rvert < 1/2$. 
Also, we assume that 
\begin{align*}
a,\; b,\; c - a,\; c - b,\; c,\; c - a - b,\; a - b \notin \mathbb{Z}. 
\end{align*}

First, we prove $(\ref{(sig0)Q})$ and $(\ref{(sig2)Q})$. 
For the purpose, we introduce two expressions for $Q$ given 
in \cite{Eb1}. 
Let $f_{i}$ $(i = 1, 2, 5, 6)$ be the functions defined by 
\begin{align*}
\y{1}{a}{b}{c}{x} &:= \f{a}{b}{c}{x}, \allowdisplaybreaks \\
\y{2}{a}{b}{c}{x} &:= \f{a}{b}{a + b + 1 - c}{1 - x}, \allowdisplaybreaks \\
\y{5}{a}{b}{c}{x} &:= x^{1 - c} \f{a + 1 - c}{b + 1 - c}{2 - c}{x}, \allowdisplaybreaks \\
\y{6}{a}{b}{c}{x} &:= (1 - x)^{c - a - b} \f{c - a}{c - b}{c + 1 - a - b}{1 - x}, 
\end{align*}
where $\displaystyle \f{a}{b}{c}{x} := \frac{\G(a) \G(b)}{\G(c)} \F{a}{b}{c}{x}$. 
Then, $Q$ can be expressed as 
\begin{align*}
Q = \frac{a b (c)_{m}}{c (a)_{k} (b)_{l}} q, 
\end{align*}
where 
\begin{align}
q &:= \q{k}{l}{m}{a}{b}{c}{x} \nonumber \\
&= \frac{\displaystyle \y{5}{a}{b}{c}{x}\, \y{1}{a + k}{b + l}{c + m}{x} - \y{1}{a}{b}{c}{x}\, \y{5}{a + k}{b + l}{c + m}{x}}{\displaystyle \y{5}{a}{b}{c}{x}\, \y{1}{a + 1}{b + 1}{c + 1}{x} - \y{1}{a}{b}{c}{x}\, \y{5}{a + 1}{b + 1}{c + 1}{x}}, \label{q1} 
\allowdisplaybreaks \\
&= \frac{(- 1)^{m + 1 - k - l}}{(c - a)_{m - k} (c - b)_{m - l}}
\frac{\displaystyle \y{6}{a}{b}{c}{x}\, \y{2}{a + k}{b + l}{c + m}{x} - \y{2}{a}{b}{c}{x}\, \y{6}{a + k}{b + l}{c + m}{x}}{\displaystyle \y{6}{a}{b}{c}{x}\, \y{2}{a + 1}{b + 1}{c + 1}{x} - \y{2}{a}{b}{c}{x}\, \y{6}{a + 1}{b + 1}{c + 1}{x}}. \label{q2}
\end{align}
These expressions $(\ref{q1})$ and $(\ref{q2})$ follow immediately 
from \cite[p.260, $(3. 5)$]{Eb1} and \cite[p.264, the expression above Theorem~$3. 8$]{Eb1}, respectively. 

Using $(\ref{q1})$, we prove $(\ref{(sig0)Q})$. 
Applying $\sigma_{0}$ to $(\ref{q1})$, we have 
\begin{align*}
\left(\sigma_{0} q \right) \c{k}{l}{m}{a}{b}{c}{x} 
&= - \frac{W (a, b, c\,; x)}{W (a + k, b + l, c + m\,; x)} \q{k}{l}{m}{a}{b}{c}{x}, 
\end{align*}
where $W (a, b, c \,; x)$ denotes the denominator of $(\ref{q1})$; namely, 
\begin{align*}
W (a, b, c\,; x) := \y{5}{a}{b}{c}{x}\, \y{1}{a + 1}{b + 1}{c + 1}{x} - \y{1}{a}{b}{c}{x}\, \y{5}{a + 1}{b + 1}{c + 1}{x}. 
\end{align*}
From the formula \cite[p.$262$, Lemma~$3. 6$]{Eb1}
\begin{align*}
W (a, b, c\,; x) = - \frac{\G (a) \G (b) \G (a + 1 - c) \G (b + 1 - c)}{\G (c) \G (1 - c)} x^{-c} (1 - x)^{c - a - b - 1}, 
\end{align*}
we obtain 
\begin{align*}
\frac{W (a, b, c\,; x)}{W (a + k, b + l, c + m\,; x)} 
= \frac{(-1)^{k + l - m} (c - a)_{m - k} (c - b)_{m - l}}{(a)_{k} (b)_{l}} x^{m} (1 - x)^{k + l - m}. 
\end{align*}
Therefore, it follows that 
\begin{align}
\left(\sigma_{0} q \right) \c{k}{l}{m}{a}{b}{c}{x} 
&= \frac{(-1)^{k + l - m - 1} (c - a)_{m - k} (c - b)_{m - l}}{(a)_{k} (b)_{l}} 
x^{m} (1 - x)^{k + l - m} \q{k}{l}{m}{a}{b}{c}{x}. \label{(sig0)q}
\end{align}
Multiplying both sides of $(\ref{(sig0)q})$ by $(a + k) (b + l) (c + m)_{-m} / \left\{(c + m) (a + k)_{-k} (b + l)_{-l} \right\}$ completes the proof of $(\ref{(sig0)Q})$. 

Using $(\ref{q1})$ and $(\ref{q2})$, we prove $(\ref{(sig2)Q})$. 
When we apply $\sigma_{2}$ to $(\ref{q1})$, the numerator becomes
\begin{align}
&\y{5}{a}{b}{a + b + 1 - c}{1 - x}\, 
\y{1}{a + k}{b + l}{a + b + 1 - c + k + l - m}{1 - x} \nonumber \\
&\quad - \y{1}{a}{b}{a + b + 1 - c}{1 - x}\, 
\y{5}{a + k}{b + l}{a + b + 1 - c + k + l - m}{1 - x}, \label{(sig2)q-nume}
\end{align}
and the denominator becomes
\begin{align}
&\y{5}{a}{b}{a + b + 1 - c}{1 - x}\, 
\y{1}{a + 1}{b + 1}{a + b + 2 - c}{1 - x} \nonumber \\
&\quad - \y{1}{a}{b}{a + b + 1 - c}{1 - x}\, 
\y{5}{a + 1}{b + 1}{a + b + 2 - c}{1 - x}. \label{(sig2)q-deno}
\end{align}
From the definitions of $f_{i}$ $(i = 1, 2, 5, 6)$, we can rewrite 
$(\ref{(sig2)q-nume})$ and $(\ref{(sig2)q-deno})$ as 
\begin{align}
&\y{6}{a}{b}{c}{x}\, \y{2}{a + k}{b + l}{c + m}{x} 
- \y{2}{a}{b}{c}{x}\, \y{6}{a + k}{b + l}{c + m}{x}, \label{(sig2)q-nume2} \allowdisplaybreaks \\
&\y{6}{a}{b}{c}{x}\, \y{2}{a + 1}{b + 1}{c + 1}{x} 
- \y{2}{a}{b}{c}{x}\, \y{6}{a + 1}{b + 1}{c + 1}{x}, \label{(sig2)q-deno2}
\end{align}
respectively. Comparing $(\ref{(sig2)q-nume2}) / (\ref{(sig2)q-deno2})$ with $(\ref{q2})$, we obtain 
\begin{align}\label{(sig2)q}
\left(\sigma_{2} q \right) \c{k}{l}{m}{a}{b}{c}{x} 
= (-1)^{k + l - m - 1} (c - a)_{m - k} (c - b)_{m - l} \q{k}{l}{m}{a}{b}{c}{x}. 
\end{align}
Multiplying both sides of $(\ref{(sig2)q})$ by 
$a b (a + b + 1 - c)_{k + l - m} / \left\{(a + b + 1 - c) (a)_{k} (b)_{l} \right\}$ completes the proof of $(\ref{(sig2)Q})$. 

Next, we prove $(\ref{(sig1)Q})$ by using the following two formulas: 
\begin{align}
\F{a}{b}{c}{x} &= (1 - x)^{- b} \F{c - a}{b}{c}{\frac{x}{x - 1}}, \label{F-3} \allowdisplaybreaks \\
\F{c - a}{b + 1}{c + 1}{\frac{x}{x - 1}} 
&= \frac{a - c}{a (1 - x)} \F{c - a + 1}{b + 1}{c + 1}{\frac{x}{x - 1}} 
+ \frac{c}{a} \F{c - a}{b}{c}{\frac{x}{x - 1}}, \label{(0,1,1)}
\end{align}
where $(\ref{F-3})$ is the Pfaff's identity, 
and $(\ref{(0,1,1)})$ is an immediate consequence of \cite[p.$14,\ (2.5)$]{Eb2}. 
Applying $(\ref{F-3})$ to both sides of $(\ref{3kou})$, we have 
\begin{align*}
\F{c - a + m - k}{b + l}{c + m}{\frac{x}{x - 1}} 
&= \Q{k}{l}{m}{a}{b}{c}{x} (1 - x)^{l - 1} \F{c - a}{b + 1}{c + 1}{\frac{x}{x - 1}} \\ 
&\quad+ \R{k}{l}{m}{a}{b}{c}{x} (1 - x)^{l} \F{c - a}{b}{c}{\frac{x}{x - 1}}. 
\end{align*}
Moreover, from $(\ref{(0,1,1)})$, we have 
\begin{align}
&\F{c - a + m - k}{b + l}{c + m}{\frac{x}{x - 1}}\nonumber \\
&= \frac{a - c}{a} (1 - x)^{l - 2} \Q{k}{l}{m}{a}{b}{c}{x} \F{c - a + 1}{b + 1}{c + 1}{\frac{x}{x - 1}} \nonumber \\
&\quad+ \left\{\frac{c}{a} (1 - x)^{l - 1} \Q{k}{l}{m}{a}{b}{c}{x} + (1 - x)^{l} \R{k}{l}{m}{a}{b}{c}{x} \right\} 
\F{c - a}{b}{c}{\frac{x}{x - 1}}. \label{F(c-a)-2}
\end{align}
On the other hand, replacing $(k, l, m, a, b, c, x)$ by $(m - k, l, m, c - a, b, c, x/(x - 1))$ in $(\ref{3kou})$, we have 
\begin{align}
\F{c - a + m - k}{b + l}{c + m}{\frac{x}{x - 1}} 
&= \Q{m - k}{l}{m}{c - a}{b}{c}{\frac{x}{x - 1}} \F{c - a + 1}{b + 1}{c + 1}{\frac{x}{x - 1}} \nonumber \\
&\quad + \R{m - k}{l}{m}{c - a}{b}{c}{\frac{x}{x - 1}} \F{c - a}{b}{c}{\frac{x}{x - 1}}. \label{Q''R''}
\end{align}
Equating the coefficients of $F (c - a + 1, b + 1, c + 1\,; x/(x - 1))$ 
in $(\ref{F(c-a)-2})$ and $(\ref{Q''R''})$ completes the proof of $(\ref{(sig1)Q})$. 

Finally, we can immediately obtain $(\ref{(sig3)Q})$ from the fact that 
$F (\alpha, \beta, \gamma\,; x)$ is symmetric with respect to the exchange of $\alpha$ and $\beta$.

\section{\bf{Proof of Lemma~\ref{R=Q'}}}
In this section, we prove Lemma~$\ref{R=Q'}$. 

Replacing $(k, l, m, a, b, c)$ by $(k - 1, l - 1, m - 1, a + 1, b + 1, c + 1)$ in $(\ref{3kou})$, we have 
\begin{align}\label{Q'R'}
\F{a + k}{b + l}{c + m}{x} = Q' \cdot \F{a + 2}{b + 2}{c + 2}{x} + R' \cdot \F{a + 1}{b + 1}{c + 1}{x}, 
\end{align}
where
\begin{align*}
Q' := \Q{k - 1}{l - 1}{m - 1}{a + 1}{b + 1}{c + 1}{x}, \quad 
R' := \R{k - 1}{l - 1}{m - 1}{a + 1}{b + 1}{c + 1}{x}. 
\end{align*}
As is well known, $F (a, b, c\,; x)$ satisfies 
\begin{align*}
\partial\, \F{a}{b}{c}{x} = \frac{a b}{c} \F{a + 1}{b + 1}{c + 1}{x}, 
\end{align*}
where $\partial := d/dx$, and is a solution of 
the hypergeometric differential equation $L_{a b c}\, y = 0$, where 
\begin{align*}
L_{a b c} := \partial^{2} + \frac{c - (a + b + 1) x}{x (1 - x)} \partial - \frac{a b}{x (1 - x)}. 
\end{align*}
Using these facts, we have 
\begin{align*}
0 =& L_{a b c}\, \F{a}{b}{c}{x} \allowdisplaybreaks \\
=& \partial^{2}\, \F{a}{b}{c}{x} + \frac{c - (a + b + 1) x}{x (1 - x)} \partial\, \F{a}{b}{c}{x} 
- \frac{a b}{x (1 - x)} \F{a}{b}{c}{x} \allowdisplaybreaks \\
=& \frac{a b (a + 1) (b + 1)}{c (c + 1)} \F{a + 2}{b + 2}{c + 2}{x} \\
&+ \frac{a b \left\{c - (a + b + 1) x \right\}}{ c x (1 - x)} \F{a + 1}{b + 1}{c + 1}{x} - \frac{a b}{x (1 - x)} \F{a}{b}{c}{x}. 
\end{align*}
Therefore, we obtain 
\begin{align*}
\F{a + 2}{b + 2}{c + 2}{x} 
&= - \frac{(c + 1) \left\{c - (a + b + 1) x \right\}}{(a + 1) (b + 1) x (1 - x)} \F{a + 1}{b + 1}{c + 1}{x} \\
&\quad + \frac{c (c + 1)}{(a + 1) (b + 1) x (1 - x)} \F{a}{b}{c}{x}. 
\end{align*}
Using this, we rewrite $(\ref{Q'R'})$ as 
\begin{align}
\F{a + k}{b + l}{c + m}{x} 
&= \left\{- \frac{(c + 1) \left\{c - (a + b + 1) x \right\}}{(a + 1) (b + 1) x (1 - x)} Q' + R' \right\} 
\F{a + 1}{b + 1}{c + 1}{x} \nonumber \\
&\quad + \frac{c (c + 1)}{(a + 1) (b + 1) x (1 - x)} Q' \cdot \F{a}{b}{c}{x}. \label{Q'R'-2} 
\end{align}
Equating the coefficients of $F (a, b, c\,; x)$ in $(\ref{3kou})$ and $(\ref{Q'R'-2})$ completes the proof of Lemma~$\ref{R=Q'}$.

\section{\bf{Proof of Theorem~$\ref{R-sym}$}}
Using Theorem~$\ref{Q-sym}$ and Lemma~$\ref{R=Q'}$, we prove Theorem~$\ref{R-sym}$. 

For any $(k, l, m) \in \mathbb{Z}^{3}$, 
let $\lambda_{\tau}$ be the rational function of $a$, $b$, $c$ and $x$ defined by
\begin{align*}
\lambda_{\tau} (\vec{z}) := \frac{c (c + 1)}{(a + 1) (b + 1) x (1 - x)}, 
\end{align*}
where $\vec{z} := (k, l, m\,; a, b, c\,; x)$. 
Then, the relation in Lemma~$\ref{R=Q'}$ can be written as
\begin{align*}
R (\vec{z}) = \lambda_{\tau} (\vec{z}) \, (\tau Q) (\vec{z}). 
\end{align*}
Also, for any $\sigma \in G$ and $(k, l, m) \in \mathbb{Z}^{3}$, 
let $\lambda_{\sigma}$ be the rational function of $a$, $b$, $c$ and $x$ satisfying
\begin{align*}
Q (\vec{z}) = \lambda_{\sigma} (\vec{z}) \, (\sigma Q) (\vec{z}). 
\end{align*}
Then, we obtain 
\begin{align}
R (\vec{z}) 
&= \lambda_{\tau} (\vec{z}) \, (\tau Q) (\vec{z}) \nonumber \\
&= \lambda_{\tau} (\vec{z}) \, (\tau \lambda_{\sigma}) (\vec{z}) \, (\tau \sigma Q) (\vec{z})  \nonumber \\
&= \lambda_{\tau} (\vec{z}) \, (\tau \lambda_{\sigma}) (\vec{z}) \, 
\left((\tau \sigma \tau^{-1}) \tau Q \right)\!(\vec{z}) \nonumber \\
&= \frac{\lambda_{\tau} (\vec{z}) \, (\tau \lambda_{\sigma}) (\vec{z})}{\left(\left(\tau \sigma \tau^{-1} \right) \lambda_{\tau} \right) (\vec{z})} 
\left((\tau \sigma \tau^{-1} ) R \right)\!(\vec{z}). \label{R=(tau.sig.tau^{-1})R}
\end{align}
This implies that $R$ has a symmetry under $\tau \sigma \tau^{-1}$ for each $\sigma \in G$; 
namely, $R$ has symmetries under the action of $\tilde{G}$. 
In particular, letting $\sigma = \sigma_{i}$ ($i = 0, 1, 2, 3$) in $(\ref{R=(tau.sig.tau^{-1})R})$, 
we can derive $(\ref{(tsig0)R})$--$(\ref{(tsig3)R})$ 
from $(\ref{(sig0)Q})$--$(\ref{(sig3)Q})$, respectively.

\medskip
\thanks{\bf{Acknowledgements}}
We are deeply grateful to Prof. Hiroyuki Ochiai for helpful comments. 
Also, we would like to thank Akihito Ebisu for his comments and suggestions.

\medskip
\begin{flushleft}
Yuka Yamaguchi\\
Faculty of Education\\
University of Miyazaki\\
1-1 Gakuen Kibanadai-nishi\\
Miyazaki 889-2192\\
JAPAN\\
y-yamaguchi@cc.miyazaki-u.ac.jp
\end{flushleft}

\end{document}